\documentclass[12pt]{article}

\usepackage{amsmath,amsthm,amsfonts,amssymb,color}

\newtheorem{theorem}{Theorem}[section]

\newtheorem{proposition}[theorem]{Proposition}
\newtheorem{lemma}[theorem]{Lemma}
\newtheorem{definition}[theorem]{Definition}

\def\cD{\mathcal{D}}

\def\cF{\mathcal{F}}
\def\cH{\mathcal{H}}

\def\cS{\mathcal{S}}

\def\bC{\mathbb{C}}

\def\E{\mathbb{E}}
\def\bR{\mathbb{R}}

\begin{document}

\title{
H\"older continuity for the Parabolic Anderson Model with
space-time homogeneous Gaussian noise}

\author{Raluca M. Balan\footnote{Corresponding author. Department of Mathematics and Statistics, University of Ottawa,
585 King Edward Avenue, Ottawa, ON, K1N 6N5, Canada. E-mail address:
rbalan@uottawa.ca. Research supported by a grant from the
Natural Sciences and Engineering Research Council of Canada.} \and Llu\'is
Quer-Sardanyons\footnote{Departament de Matem\`atiques, Universitat
Aut\`{o}noma de Barcelona, 08193 Bellaterra (Barcelona), Catalonia,
Spain. E-mail address: quer@mat.uab.cat. Research supported by the
grant MTM2015-67802P.}
 \and Jian Song\footnote{Department of Mathematics. University of Hong
Kong. E-mail address: txjsong@hku.hk.}}

\date{July 15, 2018}
\maketitle

\begin{abstract}
\noindent In this article, we consider the Parabolic Anderson Model
with constant initial condition, driven by a space-time homogeneous
Gaussian noise, with general covariance function in time and spatial
spectral measure satisfying Dalang's condition. First, we prove that
the solution (in the Skorohod sense) exists and is continuous in
$L^p(\Omega)$. Then, we show that the solution has a modification
whose sample paths are H\"older continuous in space and time, with
optimal exponents, and under the minimal condition on the spatial
spectral measure of the noise (which is the same as the condition
encountered in the case of the white noise in time).
This improves similar results which were obtained in \cite{HHNT, song17} under more restrictive conditions, and with sub-optimal exponents for H\"older continuity.

\end{abstract}

\medskip


\medskip
\vspace{1mm}


\section{Introduction}

In this article, we consider the Parabolic Anderson Model on
$\bR^d$, $d\geq 1$, with constant initial condition:
\begin{equation}
\left\{\begin{array}{rcl}
\displaystyle \frac{\partial u}{\partial t}(t,x) & = & \displaystyle \frac{1}{2}\Delta u(t,x)+u(t,x)\dot{W}(t,x), \quad t>0, x \in \bR^d, \\[2ex]
\displaystyle u(0,x) & = & 1, \quad x \in \bR^d.\\[1ex]
\end{array}\right. \label{heat} 
\end{equation}

The noise $W$ is given by a zero-mean Gaussian process
$\{W(\varphi);\varphi \in \cD(\bR^{d+1})\}$ defined on a complete
probability space $(\Omega,\cF,P)$, with covariance
$$\E[W(\varphi_1)W(\varphi_2)]=\int_{\bR^2 \times \bR^{2d}}\gamma(t-s)f(x-y)\varphi_1(t,x)\varphi_2(s,y)dxdydtds=:J(\varphi_1,\varphi_2),$$
where $\gamma:\bR \to [0,\infty]$ and $f:\bR^d \to [0,\infty]$ are continuous, symmetric, locally integrable functions, such that
$$\gamma(t) <\infty \quad \mbox{if and only if} \quad t\not=0;$$
$$f(x) <\infty \quad \mbox{if and only if} \quad x\not=0.$$
Here $\cD(\bR^{d+1})$ is the space of $C^{\infty}$-functions on $\bR^{d+1}$ with compact support.
We denote by $\cH$  the completion of $\cD(\bR^{d+1})$ with respect
to $\langle \cdot,\cdot\rangle_{\cH}$ defined by
$\langle \varphi_1,\varphi_2\rangle_{\cH}=J(\varphi_1,\varphi_2)$.

We assume that $f$ is non-negative-definite (in the sense of distributions), i.e.
$$\int_{\bR^d}(\varphi*\widetilde{\varphi})(x)f(x)dx \geq 0, \quad \mbox{for all}  \quad \varphi \in \cS(\bR^d),$$
where $\widetilde \varphi(x)=\varphi(-x)$ and $\cS(\bR^d)$ is the
space of rapidly decreasing $C^{\infty}$-functions on $\bR^d$. By
the Bochner-Schwartz theorem, there exists a tempered measure $\mu$
on $\bR^d$ such that $f=\cF \mu$, where $\cF \mu$ denotes the
Fourier transform of $\mu$ in the space $\cS_{\bC}'(\bR^d)$ of
$\bC$-valued tempered distributions on $\bR^d$, i.e.
$$\int_{\bR^d} f(x)\varphi(x)dx=\int_{\bR^d} \cF \varphi(\xi) \mu(d\xi) \quad \mbox{for all} \quad \varphi \in \cS_{\bC}(\bR^d),$$
where $\cS_{\bC}(\bR^d)$ is the space of $\bC$-valued rapidly decreasing $C^{\infty}$-functions on $\bR^d$, and $\cF \varphi(x)=\int_{\bR^d} e^{-i \xi \cdot x} \varphi(x)dx$ is the Fourier transform of $\varphi$. Here $\xi \cdot x$ denotes the scalar product in $\bR^d$.
Similarly, we assume that $\gamma$ is non-negative-definite (in the sense of distributions), and so there exists a tempered measure $\nu$ on $\bR$ such that $\gamma=\cF \nu$ in $\cS_{\bC}'(\bR)$.

We denote by $G$ the fundamental solution of the heat equation on $\bR^d$:
$$G(t,x)=\frac{1}{(2\pi t)^{d/2}}\exp \left(-\frac{|x|^2}{2t} \right), \;\; t>0,\; x\in \bR^d,$$
where $|\cdot|$ denotes the Euclidean norm. Note that $G(t,\cdot) \in \cS(\bR^d)$ and
$$\cF G(t,\cdot)(\xi)=\exp \left(-\frac{t|\xi|^2}{2} \right) \quad \mbox{for all} \quad \xi \in \bR^d.$$

Similarly to the wave equation in spatial dimension $d\leq 2$, we have the following definition of the solution (see Definition 5.1 of \cite{balan-song17}).

\begin{definition}
{\rm A square-integrable process $u=\{u(t,x); t \geq 0,x \in \bR^d\}$ with
$u(0,x)=1$ for all $x \in \bR^d$ is a (mild Skorohod) {\em solution}
of equation \eqref{heat} if $u$ has a jointly measurable modification
(denoted also by $u$) such that $\sup_{(t,x)\in [0,T] \times \bR^d}
\E|u(t,x)|^2<\infty$ for all $T>0$, and for any $t>0$ and $x \in
\bR^d$, the following equality holds in $L^2(\Omega)$:
\begin{equation}
\label{def-sol}
u(t,x)=1+\int_{0}^t \int_{\bR^d}G(t-s,x-y)u(s,y) W(\delta s, \delta y),
\end{equation}
where the stochastic integral is understood in the Skorohod sense and the process $v^{(t,x)}=\{v^{(t,x)}(s,y)=1_{[0,t]}(s)G(t-s,x-y)u(s,y);s \geq 0,y \in \bR^d\}$ is Skorohod integrable.
}
\end{definition}

We refer the reader to \cite{HHNT,balan-song17,balan-chen18} for more details regarding the concept of Skorohod solution, and to \cite{nualart06} for the background on Malliavin calculus.

From \cite{HHNT}, we know that equation \eqref{heat} has a unique Skorohod solution $u$, provided that
{\em Dalang's condition} holds:
\begin{equation}
\label{Dalang-cond}
\int_{\bR^d}\frac{1}{1+|\xi|^2}\mu(d\xi)<\infty.
\end{equation}
We should point our that this condition was considered for the first time in \cite{dalang99} for equations driven by Gaussian noise which was white in time and had the same spatial covariance structure as above.

A Feynman-Kac representation for the solution $u$ was obtained in \cite{HHNT}, under additional conditions on $\gamma$ and $\mu$. Theorem 4.9 of \cite{HHNT} shows that if the function $\gamma$ satisfies:
\begin{equation}
\label{cond-gamma}
0\leq \gamma(t) \leq C_{\beta} |t|^{-\beta} \quad \mbox{for all} \quad t \in \bR,
\end{equation}
for some $\beta \in (0,1)$, and the measure $\mu$ satisfies:
$$\int_{\bR^d}\left(\frac{1}{1+|\xi|^2} \right)^{1-\beta-\alpha} \mu(d\xi)<\infty \quad
\mbox{for some} \quad \alpha \in (0,1-\beta),$$ then for any $p \geq 2$ and $T>0$, there exists a constant $C>0$ depending on $p,T,\beta$ and $\alpha$ such that for any $t,t' \in [0,T]$ and $x,x' \in \bR^d$,
$$\|u(t,x)-u(t',x')\|_p \leq C \big(|t-t'|^{\frac{\alpha}{2}}+|x-x'|^{\alpha} \big),$$
and therefore, on any compact set $[0,T] \times K$, $u$ has a modification which is $\theta_1$-H\"older continuous in time for any $\theta_1 \in (0,\frac{\alpha}{2})$, and $\theta_2$-H\"older continuous in space for any $\theta_2 \in (0,\alpha)$. This result is based on the Feynman-Kac representation of the solution, and does not yield optimal exponents for the order of H\"older continuity (see Remark 4.10 of \cite{HHNT}).

The goal of the present article is to obtain the optimal exponents for the order of H\"older continuity of the solution $u$, under the following minimal condition on $\mu$:
\begin{equation}
\label{Holder-cond} \int_{\bR^d}\left(\frac{1}{1+|\xi|^2}
\right)^{\eta}\mu(d\xi)<\infty \quad \mbox{for some} \quad \eta \in
(0,1).
\end{equation}
We note that this condition is the same as for the white noise in time:
see Theorem 2.1 of \cite{sanz-sarra02} for the case when the initial condition is given by a bounded H\"older continuous function, and Theorem 1.8 of \cite{huang-le18} for the case when the initial condition is given by a measure. We emphasize that do not impose any constraints on the temporal covariance function $\gamma$. More precisely, in  Theorem \ref{Holder-th} below, we show that if the measure $\mu$ satisfies \eqref{Holder-cond},
then $u$ has a modification which which is $\theta_1$-H\"older continuous in time for any $\theta_1 \in (0,\frac{1-\eta}{2})$, and $\theta_2$-H\"older continuous in $x$ for any $\theta_2 \in (0,1-\eta)$.
The proof of this theorem relies on a careful analysis of the $p$-th moments of the increments of the solution $u$, which is related to the problem of continuity of $u$ in $L^p(\Omega)$. Therefore, in Theorem \ref{exist-th} below, we show that under condition \eqref{Dalang-cond}, $u$ is $L^p(\Omega)$-continuous on $[0,\infty) \times \bR^d$ for any $p \geq 2$. A special effort is dedicated to the left continuity in time, a delicate issue for the heat equation, which is usually left out in the literature.

Theorem 5.9 of \cite{song17} gives the H\"older continuity of the solution to equation \eqref{heat} under condition \eqref{Holder-cond}, with the same exponent $\theta_2 <1-\eta$ in space as in our Theorem \ref{Holder-th}, but with exponent $\theta_1 <[(1-\eta) \wedge (1-\beta)]/2$ in time, where $\beta$ is given by \eqref{cond-gamma}.

Our proofs use simplified versions of the arguments contained in \cite{balan-chen18} for the heat equation with initial condition given by a measure. But the results that we present here do not follow directly from Theorems 1.1.(a) and 1.4 of \cite{balan-chen18} whose conclusions are valid for time intervals which do not include 0. The novelty of our results stems from the fact that they are valid for time intervals of the form $[0,T]$.

This article is organized as follows. In Section \ref{section-exist}, we prove the existence and uniqueness of the Skorohod solution $u$, and the fact that $u$ is continuous in $L^p(\Omega)$. In Section \ref{section-Holder}, we show that $u$ has a modification whose sample paths are H\"older continuous with optimal exponents, as mentioned above.

\section{Existence of solution and continuity in $L^p(\Omega)$}
\label{section-exist}

In this section, we show the existence and uniqueness of the solution to equation \eqref{heat} and its continuity in $L^p(\Omega)$ on $[0,\infty) \times \bR^d$, using the method of \cite{balan-chen18} for the heat equation with initial condition given by a measure.
We note that the existence of solution to equation \eqref{heat} was established also in Theorem 3.2 of \cite{HHNT}, using a different method.

Intuitively, if it exists, the solution to equation \eqref{heat} has the Wiener chaos expansion:
\begin{equation}
\label{u-series}
u(t,x)=1+\sum_{n\geq 1} I_n(f_n(\cdot,t,x)),
\end{equation}
where $I_n$ is the multiple Wiener integral of order $n$ with respect to $W$, and
$$f_n(t_1,x_1,\ldots,t_n,x_n,t,x)=G(t-t_n,x-x_n) \ldots G(t_2-t_1,x_2-x_1)1_{\{0<t_1<\ldots<t_n<t\}}.$$
We denote $J_n(t,x)=I_n(f_n(\cdot,t,x))$ for $n\geq 1$ and $J_0(t,x)=1$.

Note that $f_n(\cdot,t,x) \in \cH_n$, where $\cH_n$ is the Wiener
chaos space of order $n$, with respect to $W$. The spaces $\cH_n,
n\geq 1$ are orthogonal. Moreover, $\E|I_n(f_n(\cdot,t,x))|^2=n!
\|\widetilde{f}_n(\cdot,t,x)\|_{\cH^{\otimes n}}^{2}$ where
$\widetilde{f}_n(\cdot,t,x)$ is the symmetrization of
$f_n(\cdot,t,x)$ given by:
$$\widetilde{f}_n(t_1,x_1,\ldots,t_n,x_n,t,x)=\frac{1}{n!}\sum_{
\rho \in S_n} f_n(t_{\rho(1)}, x_{\rho(1)}, \ldots,t_{\rho(n)},x_{\rho(n)},t,x),$$
and $S_n$ is the set of all permutations on $\{1,\ldots,n\}$.
Let $\alpha_n(t)=(n!)^2 \|\widetilde{f}_n(\cdot,t,x)\|_{\cH^{\otimes n}}^{2}$ for $n\geq 1$, and $\alpha_0(t)=1$.

As in \cite{balan-chen18}, for any $t>0$, we let
$$k(t)=\int_{\bR^d}|\cF G(t,\cdot)(\xi)|^2 \mu(d\xi),$$
$h_0(t)=1$ and for $n\geq 1$,
$$h_n(t)=\int_{0<t_1<\ldots<t_n<t} k(t_2-t_1)\ldots k(t_n-t_{n-1})k(t-t_n)dt_1 \ldots dt_n.$$
For any $\gamma>0$ and $t>0$, we let
$$H(t;\gamma)=\sum_{n\geq 0}\gamma^n h_n(t) \quad \mbox{and} \quad \widetilde{H}(t;\gamma)=\sum_{n\geq 0}\sqrt{\gamma^n h_n(t)}.$$
By Lemma 3.8 of \cite{balan-chen18}, we know that under Dalang's condition,
$H(t;\gamma)<\infty$ and $\widetilde{H}(t;\gamma)<\infty$ for any $t>0$ and $\gamma>0$.
The following result is a particular case of Lemma 3.4 of \cite{balan-chen18}.

\begin{lemma}
\label{sup-lemma}
For any $t>0$,
$$\sup_{\eta \in \bR^d}\int_{\bR^d} |\cF G(t,\cdot)(\xi+\eta)|^2 \mu(d\xi)=k(t).$$
\end{lemma}

We denote by $\|\cdot\|_p$ the norm in $L^p(\Omega)$. We will use frequently the fact that the $\|\cdot\|_p$-norms are equivalent on a {\em fixed} Wiener chaos space $\cH_n$. More precisely,
\begin{equation}
\label{equivalence-p-norms}
\|F\|_p \leq (p-1)^{n/2} \|F\|_2 \quad \mbox{for all} \ F \in \cH_n
\end{equation}
(see the last line of page 62 of \cite{nualart06}).

\begin{theorem}
\label{exist-th}
Under condition \eqref{Dalang-cond}, the series on the right-hand side of \eqref{u-series} converges in $L^2(\Omega)$ and the process $u$ given by \eqref{u-series} is the unique solution to equation \eqref{heat}. Moreover, for any $p\geq 2$, $u$ is $L^p(\Omega)$-continuous on $[0,\infty) \times \bR^d$, and
$$\sup_{(t,x) \in [0,T] \times \bR^d}\|u(t,x)\|_p<\infty, \quad for \ all \quad T>0.$$
\end{theorem}

\begin{proof}
{\em Step 1. (Summability of the series)} We first show that the series on the right-hand side of \eqref{u-series} converges in $L^2(\Omega)$, which is equivalent to:
$$\sum_{n\geq 0} \frac{1}{n!} \alpha_n(t)<\infty.$$
By relation (4.15) in \cite{balan-song17},
$$\alpha_n(t) \leq \Gamma_t^n \int_{[0,t]^n} \psi_{n}({\bf t},{\bf t})d{\bf t},$$
where ${\bf t}=(t_1,\ldots,t_n)$, $\Gamma_t=2 \int_0^t \gamma(s)ds$,
\begin{align}
\nonumber
\psi_n({\bf t},{\bf t}) & =\int_{\bR^{nd}} |\cF G(u_1,\cdot)(\xi_1)|^2 \ldots |\cF G(u_{n-1},\cdot)(\xi_1+\ldots+\xi_{n-1})|^2 \\
\label{def-psi-t} &  \quad \quad \quad \times |\cF
G(u_{n},\cdot)(\xi_1+\ldots+\xi_n)|^2 \mu(d\xi_1) \ldots
\mu(d\xi_n),
\end{align}
$u_j=t_{\rho(j+1)}-t_{\rho(j)}$ for $j=1,\ldots,n$,
$t_{\rho(n+1)}=t$, and $\rho$ is the permutation of $\{1,\ldots,n\}$
such that $t_{\rho(1)}<\ldots<t_{\rho(n)}$. Note that
$$\psi_n({\bf t},{\bf t}) \leq \prod_{j=1}^{n} \left(\sup_{\eta \in \bR^d} \int_{\bR^d} |\cF G(u_j,\cdot)(\xi_j+\eta)|^2 \mu(d\xi_j) \right)=\prod_{j=1}^{n}k(u_j),$$
where the last equality follows by Lemma \ref{sup-lemma}.
Hence,
\begin{align*}
\alpha_n(t) & \leq \Gamma_t^n \sum_{\rho \in S_n} \int_{0<t_{\rho(1)}<\ldots<t_{\rho(n)}<t} \prod_{j=1}^{n} k(t_{\rho(j+1)}-t_{\rho(j)}) d{\bf t}\\
&= \Gamma_t^n n! \int_{0<t_1<\ldots<t_n<t} \prod_{j=1}^{n}k(t_{j+1}-t_j) d{\bf t}=\Gamma_t^n n! h_n(t),
\end{align*}
and
$$\sum_{n\geq 0}\frac{1}{n!} \alpha_n(t) \leq \sum_{n\geq 0} \Gamma_t^n h_n(t)=H(t;\Gamma_t)<\infty.$$

{\em Step 2. (Existence of solution)}
The fact that the process $u$ given by \eqref{u-series} is a solution to equation \eqref{heat} follows exactly as in the proof of Theorem 5.2 of \cite{balan-song17}. This argument uses the fact that $u$ is $L^2(\Omega)$-continuous, which will be proved in {\em Step 5} below.

{\em Step 3. (Uniqueness of solution)} This follows as in the proof of Theorem 6.1 of \cite{balan-song17}.

{\em Step 4. (Uniform boundedness of $p$-th moments)}
By Minkovski's inequality and \eqref{equivalence-p-norms},
\begin{align*}
\|u(t,x)\|_p & \leq \sum_{n\geq 0} (p-1)^{n/2} \|J_n(t,x)\|_{2} =\sum_{n\geq 1} (p-1)^{n/2} \left(
\frac{1}{n!}\alpha_n(t)\right)^{1/2} \\
&\leq \sum_{n\geq 0} (p-1)^{n/2} \Gamma_t^{n/2} \sqrt{h_n(t)}=\widetilde{H}\big(t;(p-1)\Gamma_t\big).
\end{align*}
Since the function $\Gamma_t$ is non-decreasing in $t$ and $\widetilde{H}(t;\gamma)$ is non-decreasing in $t$ and $\gamma$, we infer that
$\|u(t,x)\|_p \leq \widetilde{H}\big(T;(p-1)\Gamma_T\big)<\infty$ for any $(t,x) \in [0,T] \times \bR^d$.

{\em Step 5. (Continuity in $L^p(\Omega)$)} This follows as in the proof of Theorem 7.1 of \cite{balan-song17}, provided we show that $J_n$ is $L^p(\Omega)$-continuous, for any $n\geq 1$.

We start with the right-continuity in time of $J_n$.  We will prove that for any $t \geq 0$,
$$\lim_{h \downarrow 0}\|J_n(t+h,x)-J_n(t,x)\|_p=0 \quad \mbox{uniformly in} \ x \in \bR^d.$$
By \eqref{equivalence-p-norms}, for any $t \geq 0$, $x \in \bR^d$ and $h>0$
\begin{align}
\nonumber
\|J_n(t+h,x)-J_n(t,x)\|_p  & \leq (p-1)^{n/2}\|J_n(t+h,x)-J_n(t,x)\|_2  \\
& \leq (p-1)^{n/2}\left[\frac{2}{n!}\big(A_n(t,h)+B_n(t,h) \big) \right]^{1/2},
\label{norm-p-J}
\end{align}
where
\begin{eqnarray*}
A_n(t,h)&=&(n!)^2 \|\widetilde{f}_n(\cdot,t+h,x)1_{[0,t]^{n}}-\widetilde{f}_n(\cdot,t,x) \|_{\cH^{\otimes n}}^{2}\\
B_n(t,h)&=& (n!)^2 \|\widetilde{f}_n(\cdot,t+h,x)1_{[0,t+h]^{n}\verb2\2 [0,t]^n} \|_{\cH^{\otimes n}}^{2}.
\end{eqnarray*}
Note that $A_n(t,h)$ and $B_n(t,h)$ do not depend on $x$.

We first treat $A_n(t,h)$. Using relations (7.6) and (7.7) of \cite{balan-song17}, we see that
\begin{equation}
\label{estimate-A}
A_n(t,h) \leq \Gamma_t^n \int_{[0,t]^n} \psi_{t,h}^{(n)}({\bf t},{\bf t})d{\bf t},
\end{equation}
where ${\bf t}=(t_1,\ldots,t_n)$,
\begin{align}
\nonumber
\psi_{t,h}^{(n)}({\bf t},{\bf t})&= \int_{\bR^{nd}} |\cF G(u_1,\cdot)(\xi_1)|^2 \ldots |\cF G(u_{n-1},\cdot)(\xi_1+\ldots+\xi_{n-1})|^2 \\
\label{def-psi-th} & \qquad \quad \times |\cF
[G(u_n+h,\cdot)-G(u_n,\cdot)](\xi_1+\ldots+\xi_n)|^2
\mu(d\xi_1)\ldots \mu(d\xi_n),
\end{align}
and $u_1, \ldots,u_n$ are the same as in {\em Step 1} above. We use the fact that
\begin{align*}
& |\cF [G(u_n+h,\cdot)-G(u_n,\cdot)](\xi_1+\ldots+\xi_n)|^2
= \left(e^{-(u_n+h)|\xi_1+\ldots+\xi_n|^2/2}-e^{-u_n|\xi_1+\ldots+\xi_n|^2/2}\right)^2  \\
& \quad \quad \quad \quad \quad
=|\cF G(u_n,\cdot)(\xi_1+\ldots+\xi_n)|^2\left( 1-e^{-h|\xi_1+\ldots+\xi_n|^2/2}\right)^2
\end{align*}
converges to $0$ when $h \to 0$, and is bounded by $|\cF G(u_n,\cdot)(\xi_1+\ldots+\xi_n)|^2$ (since $h>0$) . Note that
$$\int_{\bR^{nd}} |\cF G(u_1,\cdot)(\xi_1)|^2 \ldots |\cF G(u_n,\cdot)(\xi_1+\ldots+\xi_n)|^2 \mu(d\xi_1)\ldots \mu(d\xi_n) \leq \prod_{j=1}^{n}k(u_j).$$
By the dominated convergence theorem, $\psi_{t,h}^{(n)}({\bf t},{\bf t}) \to 0$ as $h \to 0$. Moreover,
$\psi_{t,h}^{(n)}({\bf t},{\bf t})$ is bounded by the function $\prod_{j=1}^{n}k(u_j)$, which is integrable on $[0,t]^n$, its integral being equal to $n! h_n(t)$. By the dominated convergence theorem, $A_n(t,h) \to 0$ as $h \to 0$.

Next, we treat $B_n(t,h)$. Let $D_{t,h}=[0,t+h]^n \verb2\2 [0,t]^n$. Using relation (7.10) of \cite{balan-song17}, we see that
\begin{equation}
\label{estimate-B}
B_n(t,h) \leq \Gamma_{t+h}^n \int_{[0,t+h]^n} \gamma_{t,h}^{(n)}({\bf t},{\bf t}) 1_{D_{t,h}}({\bf t}) d{\bf t},
\end{equation}
where
\begin{align}
\nonumber
\gamma_{t,h}^{(n)}({\bf t},{\bf t}) & =\int_{\bR^{nd}} |\cF G(u_1,\cdot)(\xi_1)|^2 \ldots |\cF G(u_{n-1},\cdot)(\xi_1+\ldots+\xi_{n-1})|^2 \\\
\label{def-gamma-th}
& \quad \quad \quad |\cF G(u_n+h)(\xi_1+\ldots+\xi_n)|^2 \mu(d\xi_1) \ldots \mu(d\xi_n).
\end{align}
We use the fact that
$$|\cF G(u_n+h,\cdot)(\xi_1+\ldots+\xi_n)|^2=|\cF G(u_n,\cdot)(\xi_1+\ldots+\xi_n)|^2 e^{-h|\xi_1+\ldots+\xi_n|^2}$$
converges to $|\cF G(u_n,\cdot)(\xi_1+\ldots+\xi_n)|^2$ as $h \to 0$, and is bounded by $|\cF G(u_n,\cdot)(\xi_1+\ldots+\xi_n)|^2$ (since $h>0$). By the dominated convergence theorem,
\begin{align*}
\lim_{h \to 0}\gamma_{t,h}^{(n)}({\bf t},{\bf t}) & = \int_{\bR^{nd}} |\cF G(u_1,\cdot)(\xi_1)|^2 \ldots |\cF G(u_{n-1},\cdot)(\xi_1+\ldots+\xi_{n-1})|^2 \\
& \quad \quad \quad \times |\cF G(u_n,\cdot)(\xi_1+\ldots+\xi_n)|^2
\mu(d\xi_1) \ldots \mu(d\xi_n).
\end{align*}
But $1_{D_{t,h}} \to 0$ as $h \to 0$ and $\gamma_{t,h}^{(n)}({\bf
t},{\bf t})1_{D_{t,h}} $ is bounded by the function
$\prod_{j=1}^{n}k(u_j)$, which is integrable on $[0,t]^n$. By the
dominated convergence theorem, $B_n(t,h) \to 0$ as $h \to 0$.

We consider now the left-continuity in time of $J_n$.  We will prove that for any $t>0$,
$$\lim_{h \downarrow 0}\|J_n(t-h,x)-J_n(t,x)\|_p=0 \quad \mbox{uniformly in} \ x \in \bR^d.$$

We argue as in Step 2 of the proof of Lemma B.3 of \cite{balan-chen18}.
By \eqref{equivalence-p-norms}, for any $t> 0$, $x \in \bR^d$ and $h \in (0,t)$
\begin{align}
\nonumber
\|J_n(t-h,x)-J_n(t,x)\|_p  & \leq (p-1)^{n/2}\|J_n(t-h,x)-J_n(t,x)\|_2  \\
& \leq (p-1)^{n/2}\left[\frac{2}{n!}\big(A_n'(t,h)+B_n'(t,h) \big) \right]^{1/2},
\label{norm-p-J}
\end{align}
where
\begin{eqnarray*}
A_n'(t,h)&=&(n!)^2 \|\widetilde{f}_n(\cdot,t-h,x)-\widetilde{f}_n(\cdot,t,x)1_{[0,t-h]^{n}} \|_{\cH^{\otimes n}}^{2}\\
B_n'(t,h)&=& (n!)^2 \|\widetilde{f}_n(\cdot,t,x)1_{[0,t]^{n}\verb2\2 [0,t-h]^n} \|_{\cH^{\otimes n}}^{2}.
\end{eqnarray*}
Note that $A_n'(t,h)$ and $B_n'(t,h)$ do not depend on $x$.

We first treat $A_n'(t,h)$. Note that
\begin{align*}
A_n'(t,h)& =\int_{[0,t-h]^{2n}} \prod_{j=1}^{n}\gamma(t_j-s_j)
\psi_{t,h}^{(n)'}({\bf t},{\bf s})d{\bf t} d{\bf s}
 \leq \Gamma_{t-h}^n \int_{[0,t-h]^n} \psi_{t,h}^{(n)'}({\bf t},{\bf t})d{\bf t},
\end{align*}
where
\begin{align*}
\psi_{t,h}^{(n)'}({\bf t},{\bf s})&= \int_{\bR^{nd}} \cF [g_{\bf t}^{(n)}(\cdot,t,x)-g_{\bf t}^{(n)}(\cdot,t-h,x)](\xi_1,\ldots,\xi_n)\\
& \quad \quad \quad \times \overline{\cF [g_{\bf
s}^{(n)}(\cdot,t,x)-g_{\bf
s}^{(n)}(\cdot,t-h,x)](\xi_1,\ldots,\xi_n)} \, \mu(d\xi_1) \ldots
\mu(d\xi_n)
\end{align*}
and $g_{\bf t}^{(n)}(\cdot,t,x)=n! \widetilde{f}_n(t_1,\cdot,\ldots,t_n,\cdot,t,x)$.
Let $(t_1,\ldots,t_n) \in [0,t-h]^n$ and $\rho$ be the permutation of $1,\ldots,n$ such that
$t_{\rho(1)}<\ldots<t_{\rho(n)}$. Let $u_j=t_{\rho(j+1)}-t_{\rho(j)}$ for any $j=1,\ldots,n-1$. Note that
\begin{align*}
&\cF g_{\bf t}^{(n)}(\cdot,t-h,x)(\xi_1,\ldots,\xi_n)\\
&\quad  =e^{-i (\xi_1+\ldots+\xi_n)\cdot x} \cF G(u_1,\cdot)(\xi_{\rho(1)}) \ldots \cF G(u_{n-1},\cdot)(\xi_{\rho(1)}+\ldots+\xi_{\rho(n-1)}) \\
& \qquad \quad \times \cF
G(t-h-t_{\rho(n)},\cdot)(\xi_{\rho(1)}+\ldots+\xi_{\rho(n)})
\end{align*}
and
\begin{align*}
& \cF g_{\bf t}^{(n)}(\cdot,t,x)(\xi_1,\ldots,\xi_n)\\
&\quad =e^{-i (\xi_1+\ldots+\xi_n)\cdot x} \cF G(u_1,\cdot)(\xi_{\rho(1)}) \ldots \cF G(u_{n-1},\cdot)(\xi_{\rho(1)}+\ldots+\xi_{\rho(n-1)}) \\
&\qquad \quad \times  \cF
G(t-t_{\rho(n)},\cdot)(\xi_{\rho(1)}+\ldots+\xi_{\rho(n)}).
\end{align*}
Let $u_n=t-t_{\rho(n)}$. Note that $u_n-h>0$ since
$t_{\rho(n)}<t-h$. It follows that
\begin{align*}
\psi_{t,h}^{(n)'}({\bf t},{\bf t})&= \int_{\bR^{nd}} |\cF G(u_1,\cdot)(\xi_1)|^2 \ldots |\cF G(u_{n-1},\cdot)(\xi_1+\ldots+\xi_{n-1})|^2 \\
\label{def-psi-th}
& \quad \quad \times |\cF [G(u_n-h,\cdot)-G(u_n,\cdot)](\xi_1+\ldots+\xi_n)|^2 \mu(d\xi_1)\ldots \mu(d\xi_n) \\
& \leq \prod_{j=1}^{n-1} \left(\sup_{\eta \in \bR^d} \int_{\bR^d} |\cF G(u_j,\cdot)(\xi_j+\eta)|^2 \mu(d\xi_j) \right) \\
& \quad \quad \times  \sup_{\eta \in \bR^d} \int_{\bR^d} |\cF G(u_n-h,\cdot)(\xi_n+\eta)-\cF G(u_n,\cdot)(\xi_n+\eta)|^2 \mu(d\xi_n) \\
& =\prod_{j=1}^{n-1} k(u_j) \cdot  \sup_{\eta \in \bR^d} \int_{\bR^d} |\cF G(u_n-h,\cdot)(\xi_n+\eta)-\cF G(u_n,\cdot)(\xi_n+\eta)|^2 \mu(d\xi_n).
\end{align*}
Using the inequality $(1-e^{-x})^2 \leq 1-e^{-x} \leq \min(x,1)$, we obtain:
\begin{align*}
|\cF G(u_n-h,\cdot)(\xi)-\cF G(u_n,\cdot)(\xi)|^2&
=e^{-(u_n-h)|\xi|^2}\big(1-e^{-h|\xi|^2/2} \big)^2 \\
& \leq e^{-(u_n-h)|\xi|^2}  \min \left( \frac{h|\xi|^2}{2},1\right)\\
&=\min\left( e^{-(u_n-h)|\xi|^2}  \frac{h|\xi|^2}{2}, e^{-(u_n-h)|\xi|^2}  \right).
\end{align*}
We bound separately the two terms appearing in the minimum above. For the first term, we use inequality (B.27) of \cite{balan-chen18} which we recall below: for any $A>0$ and $x \geq 0$,
$$\exp \left(-A x^2 \right) x^{2}  \leq \frac{2}{e} A^{-1} \exp \left(-\frac{A}{2} x^2 \right).$$
We obtain:
$$ \frac{h}{2} e^{-(u_n-h)|\xi|^2}  |\xi|^2 \leq \frac{1}{e} \cdot \frac{h}{u_n-h}e^{-(u_n-h)|\xi|^2/2}.$$
For the second term, we use the fact that $e^{-(u_n-h)|\xi|^2} \leq e^{-(u_n-h)|\xi|^2/2}$. Hence,
$$|\cF G(u_n-h,\cdot)(\xi)-\cF G(u_n,\cdot)(\xi)|^2 \leq \min \left( \frac{1}{e} \cdot \frac{h}{u_n-h},1\right) e^{-(u_n-h)|\xi|^2/2}$$
and
\begin{align*}
& \sup_{\eta \in \bR^d}\int_{\bR^d}|\cF G(u_n-h,\cdot)(\xi+\eta)-\cF G(u_n,\cdot)(\xi+\eta)|^2  \\
 & \quad \quad \leq \min \left( \frac{1}{e} \cdot \frac{h}{u_n-h},1\right) \sup_{\eta \in \bR^d}\int_{\bR^d}e^{-(u_n-h)|\xi+\eta|^2/2}\mu(d\xi) \\
& \quad \quad  = \min \left( \frac{1}{e} \cdot
\frac{h}{u_n-h},1\right)  k \left(\frac{u_n-h}{2} \right),
\end{align*}
using Lemma \ref{sup-lemma} for the last equality. It follows that
\begin{align*}
A_n'(t,h) & \leq \Gamma_{t-h}^n n! \int_{0<t_1<\ldots<t_{n}<t-h} \prod_{j=1}^{n-1}k(t_{j+1}-t_j) \min \left( \frac{1}{e} \cdot \frac{h}{t-t_n-h},1\right) k\left( \frac{t-t_n-h}{2}\right) d{\bf t} \\
&= \Gamma_{t-h}^n n! \int_0^{t-h} h_{n-1}(t_n) \min \left( \frac{1}{e} \cdot \frac{h}{t-t_n-h},1\right) k\left( \frac{t-t_n-h}{2}\right) dt_n \\
&= \Gamma_{t-h}^n n! \int_0^{t-h} h_{n-1}(t-h-s) \min \left(\frac{h}{es},1\right) k\left( \frac{s}{2}\right) ds\\
& \leq \Gamma_{t}^n n! \, h_{n-1}(t)\int_0^t  \min \left(\frac{h}{es},1\right) k\left( \frac{s}{2}\right)  ds.
\end{align*}
By the dominated convergence theorem, the last integral above converges to $0$ as $h \to 0$ since the integrand converges to $0$ when $h \to 0$ and is bounded by $k(s/2)$, and by \eqref{Dalang-cond},
$$\int_0^t k\left(\frac{s}{2}\right)ds=2\int_0^{t/2} k(s)ds=2\int_0^{t/2} \int_{\bR^d}|\cF G(s,\cdot)(\xi)|^2 \mu(d\xi) ds<\infty.$$

Next, we treat $B_n'(t,h)$. Let $D_{t,h}'=[0,t]^n \verb2\2 [0,t-h]^n$. Then
\begin{align*}
B_n'(t,h) & =\int_{[0,t-h]^n} \prod_{j=1}^{n}\gamma(t_j-s_j) \psi_n({\bf t},{\bf s})1_{D_{t,h}'}({\bf t})
1_{D_{t,h}'}({\bf s})d{\bf t} d{\bf s} \\
& \leq \Gamma_{t-h}^n  \int_{[0,t-h]^n} \psi_n({\bf t},{\bf t})1_{D_{t,h}'}({\bf t}) d{\bf t},
\end{align*}
where $\psi_n({\bf t},{\bf t})$ is given by \eqref{def-psi-t}. The last integral converges to $0$ when $h \to 0$ by the dominated convergence theorem, since $D_{t,h}' \to \emptyset$ when $h \to 0$.

Finally, we examine the continuity in space of $J_n$. By \eqref{equivalence-p-norms}, for any $x\in \bR^d$ and $z \in \bR^d$,
\begin{equation}
\label{norm-p-J2}
\|J_n(t,x+z)-J_n(t,x)\|_p  \leq (p-1)^{n/2}\|J_n(t,x+z)-J_n(t,x)\|_2 \leq (p-1)^{n/2}\left( \frac{1}{n!}C_n(t,z) \right)^{1/2},
\end{equation}
where
 $$C_n(t,z)=(n!)^2 \|\widetilde{f}_n(\cdot,t,x+z)-\widetilde{f}_n(\cdot,t,x) \|_{\cH^{\otimes n}}^{2}.$$
By relation (7.14) of \cite{balan-song17},
\begin{equation}
\label{estimate-C}
C_n(t,z) \leq \Gamma_t^n \int_{[0,t]^n}\psi_{t,z}^{(n)}({\bf t},{\bf t})d{\bf t},
\end{equation}
where
\begin{align}
\nonumber
\psi_{t,z}^{(n)}({\bf t},{\bf t}) &= \int_{\bR^{nd}}|\cF G(u_1,\cdot)(\xi_1)|^2 \ldots |\cF G(u_{n-1},\cdot)(\xi_{n-1})|^2 \\
\label{def-psi-tz} & \qquad \quad \times |\cF
G(u_n,\cdot)(\xi_1+\ldots+\xi_n)|^2 |1-e^{-i
(\xi_1+\ldots+\xi_n)\cdot z}|^2 \mu(d\xi_1) \ldots \mu(d\xi_n)
\end{align}
Note that  $|1-e^{-i (\xi_1+\ldots+\xi_n)\cdot z}|^2$ converges to $0$ when $|z| \to 0$ and is bounded by $2$. By the dominated convergence theorem, $C_n(t,z) \to 0$ when $|z| \to 0$.

\end{proof}

\section{H\"older continuity}
\label{section-Holder}

In this section, we prove that under condition \eqref{Holder-cond}, the solution $u$ is H\"older continuous.
The proof of this theorem is based on a preliminary result regarding the Fourier transform of the fundamental solution of the heat equation, which is similar to Proposition 7.4 of \cite{conus-dalang09} that deals with the fundamental solution of the wave equation.

\begin{proposition}
\label{t-increm-G}
If $\mu$ satisfies \eqref{Holder-cond}, then for any $t>0$, $h>0$ and $z \in \bR^d$,
\begin{equation}
\label{sup-1}
\sup_{\eta \in \bR^d} \int_{\bR^d} |\cF G(t+h,\cdot)(\xi+\eta)-\cF G(t,\cdot)(\xi+\eta)|^2 \mu(d\xi) \leq C h^{\theta} t^{-\theta} k\Big(\frac{t}{2}\Big),
\end{equation}
\begin{equation}
\label{sup-2}
\sup_{\eta \in \bR^d} \int_{\bR^d} |e^{-i(\xi+\eta)\cdot z}-1|^2 \, |\cF G(t,\cdot)(\xi+\eta)|^2 \mu(d\xi) \leq C |z|^{2\theta} t^{-\theta}  k\Big(\frac{t}{2}\Big),
\end{equation}
where $\theta=1-\eta$, $C>0$ is a constant depending only on
$\theta$, and we recall that $k(t)=\int_{\bR^d} |\cF
G(t,\cdot)(\xi)|^2\mu(d\xi)$.
\end{proposition}

\begin{proof} We use the same estimates as in the proof of Theorem 1.4 of \cite{balan-chen18}. More precisely, we will use inequality (5.2) of \cite{balan-chen18} which we recall below: for any $A>0$,
\begin{equation}
\label{ineq1}
\exp(-Ax^2) x^{2\theta} \leq C_{\theta} A^{-\theta} \exp \left( -\frac{A}{2}x^2\right) \quad \mbox{for all} \quad x \geq 0,
\end{equation}
where $C_{\theta}>0$ is a constant depending on $\theta$. Using the fact that $(1-e^{-x})^2 \leq (1-e^{-x})^{\theta} \leq x^{\theta}$ for any $x\geq 0$, it follows that
$\Big(1-e^{-h|\xi|^2/2} \Big)^2 \leq 2^{-\theta} h^{\theta}|\xi|^{2\theta}$. Hence, for any $\xi \in \bR^d$,
\begin{align*}
|\cF G(t+h,\cdot)(\xi)-\cF G(t,\cdot)(\xi)|^2& =e^{-t|\xi|^2} \Big(1-e^{-h|\xi|^2/2} \Big)^2 \leq 2^{-\theta}h^{\theta} e^{-t|\xi|^2}|\xi|^{2\theta} \\
& \leq 2^{-\theta} C_{\theta} h^{\theta}t^{-\theta} e^{-t|\xi|^2/2}=2^{-\theta} C_{\theta} h^{\theta}t^{-\theta} |\cF G(t/2,\cdot)(\xi)|^2
\end{align*}
It follows that
\begin{align*}
& \sup_{\eta \in \bR^d} \int_{\bR^d} |\cF G(t+h,\cdot)(\xi+\eta)-\cF G(t,\cdot)(\xi+\eta)|^2 \mu(d\xi)  \\
& \quad \quad \quad \leq 2^{-\theta} C_{\theta}
h^{\theta}t^{-\theta} \sup_{\eta \in \bR^d} \int_{\bR^d} |\cF
G(t/2,\cdot)(\xi+\eta)|^2 \mu(d\xi)\\
& \quad \quad \quad =2^{-\theta} C_{\theta} h^{\theta}t^{-\theta}
k(t/2),
\end{align*}
where the last equality is due to Lemma \ref{sup-lemma}. This proves \eqref{sup-1}.

To prove \eqref{sup-2}, we use the inequality: for any $x \in \bR$,
$$|1-e^{ix}|^2 =2(1-\cos x)^2 \leq K_{\theta} |x|^{2\theta},$$
where $K_{\theta}>0$ is a constant depending on $\theta$. Combining this with Cauchy-Schwarz inequality $|\xi \cdot z| \leq |\xi| |z|$ and inequality \eqref{ineq1}, we obtain that for any $\xi \in \bR^d$,
$$|e^{-i \xi \cdot z}-1|^2 \,|\cF G(t,\cdot)(\xi)|^2 \leq K_{\theta}|z|^{2\theta} |\xi|^{2\theta}  e^{-t|\xi|^2} \leq  K_{\theta} C_{\theta}|z|^{2\theta} t^{-\theta}e^{-t|\xi|^2/2}.$$
Hence
\begin{align*}
\sup_{\eta \in \bR^d} \int_{\bR^d} |e^{-i(\xi+\eta)\cdot z}-1|^2 \, |\cF G(t,\cdot)(\xi+\eta)|^2 \mu(d\xi) \leq
 K_{\theta} C_{\theta}|z|^{2\theta} t^{-\theta} \sup_{\eta \in \bR^d}\int_{\bR^d}e^{-t|\xi+\eta|^2/2} \mu(d\xi)
\end{align*}
and relation \eqref{sup-2} follows by Lemma \ref{sup-lemma}, exactly as above.
\end{proof}

From Lemma 5.1 of \cite{balan-chen18}, we know that condition \eqref{Holder-cond} is equivalent to
\begin{equation}
\label{cond-k}
\int_0^1 \frac{k(s)}{s^{1-\eta}} ds<\infty.
\end{equation}

\begin{theorem}
\label{Holder-th}
Under condition \eqref{Holder-cond}, for any $p \geq 2$ and $T>0$, there exists a constant $C>0$ depending on $p,T$ and $\eta$ such that for any $t,t' \in [0,T]$ and $x,x' \in \bR^d$,
$$\|u(t,x)-u(t',x')\|_p \leq C \big(|t-t'|^{\frac{1-\eta}{2}}+|x-x'|^{1-\eta} \big).$$
Consequently, for any $T>0$ and for any compact set $K \subset \bR^d$, the solution $\{u(t,x); t \in [0,T],x \in K\}$ to equation \eqref{heat} has a modification which is $\theta_1$-H\"older continuous in time for any $\theta_1 \in (0,\frac{1-\eta}{2})$, and $\theta_2$-H\"older continuous in space for any $\theta_2 \in (0,1-\eta)$.
\end{theorem}

\begin{proof}
We argue as in the proof of Theorem 8.3 of \cite{balan-song17}. Let $\theta=1-\eta$.

{\em Step 1. (increments in time)}
Assume that $t'>t$ and let $h=t'-t$. By Minkowski's inequality and \eqref{norm-p-J},
\begin{align}
\label{A-B} \|u(t+h,x)-u(t,x)\|_p  & \leq \sum_{n\geq
1}\|J_n(t+h,x)-J_n(t,x)\|_p \nonumber \\
& \leq \sum_{n\geq 1}
(p-1)^{n/2}\left[\frac{2}{n!}\big(A_n(t,h)+B_n(t,h) \big)
\right]^{1/2}.
\end{align}

We first treat $A_n(t,h)$. Recall estimate \eqref{estimate-A} for $A_n(t,h)$ that we obtained above. Using definition \eqref{def-psi-th} of $\psi_{t,h}^{(n)}({\bf t},{\bf t})$, followed by Lemma \ref{sup-lemma} and \eqref{sup-1}, we see that
\begin{align*}
\psi_{t,h}^{(n)}({\bf t},{\bf t}) & \leq \prod_{j=1}^{n-1} \left(\sup_{\eta \in \bR^d} \int_{\bR^d} |\cF G(u_j,\cdot)(\xi_j+\eta)|^2 \mu(d\xi_j)\right)\\
&\quad \quad \times \sup_{\eta \in \bR^d} \int_{\bR^d} |\cF G(u_n+h,\cdot)(\xi_n+\eta)-\cF G(u_n,\cdot)(\xi_n+\eta)|^2 \mu(d\xi_n) \\
&\leq C h^{\theta}\prod_{j=1}^{n-1} k(u_j) u_n^{-\theta} k(u_n/2).
\end{align*}
We let $\rho_t=\int_0^t s^{-\theta}k(s)ds$. Then
\begin{align*}
A_n(t,h)&\leq C h^{\theta}  \Gamma_t^n \sum_{\rho \in S_n} \int_{0<t_{\rho(1)}<\ldots<t_{\rho(n)}<t}\prod_{j=1}^{n-1}k(t_{\rho(j+1)}-t_{\rho(j)}) (t-t_{\rho(n)})^{-\theta}
k \left(\frac{t-t_{\rho(n)}}{2} \right) d{\bf t} \\
&= C h^{\theta}  \Gamma_t^n n! \int_{0<t_1<\ldots<t_n<t} \prod_{j=1}^{n-1}k(t_{j+1}-t_j) (t-t_n)^{-\theta} k\left( \frac{t-t_n}{2}\right)d{\bf t} \\
&=C h^{\theta}  \Gamma_t^n n! \int_0^t h_{n-1}(t_n) (t-t_n)^{-\theta}  k\left( \frac{t-t_n}{2}\right)dt_n \\
& \leq C h^{\theta}  \Gamma_t^n n!  h_{n-1}(t) \int_0^t (t-t_n)^{-\theta}  k\left( \frac{t-t_n}{2}\right)dt_n \\
&=C h^{\theta}  \Gamma_t^n n!  h_{n-1}(t)  \rho_{t/2}.
\end{align*}

From here, we infer that
\begin{align*}
\sum_{n\geq 1} (p-1)^{n/2}\left(\frac{1}{n!}A_n(t,h) \right)^{1/2} & \leq C h^{\theta/2} \sqrt{(p-1)\Gamma_t \rho_{t/2}} \sum_{n\geq 1} \sqrt{(p-1)^{n-1}\Gamma_t^{n-1} h_{n-1}(t)} \\
& = C h^{\theta/2} \sqrt{(p-1)\Gamma_t \rho_{t/2}} \, \widetilde{H}\big(t; (p-1)\Gamma_t\big).
\end{align*}
Since $\Gamma_t$ and $\rho_t$ are non-decreasing in $t$, and $\widetilde{H}(t;\gamma)$ is non-decreasing in both $t$ and $\gamma$, it follows that for any $T>0$, there exists a constant $C_T^{(1)}>0$ depending on $T$ such that
\begin{equation}
\label{step-A}
\sum_{n\geq 1} (p-1)^{n/2}\left(\frac{1}{n!}A_n(t,h) \right)^{1/2} \leq C_T^{(1)} h^{\theta/2}.
\end{equation}

Next, we treat $B_n(t,h)$. Recall estimate \eqref{estimate-B} for $B_n(t,h)$ that we obtained above. Using definition \eqref{def-gamma-th} of $\gamma_{t,h}^{(n)}({\bf t},{\bf t})$, followed by Lemma \ref{sup-lemma}, we see that
\begin{align*}
\gamma_{t,h}^{(n)}({\bf t},{\bf t})
& \leq \prod_{j=1}^{n-1} \left(\sup_{\eta \in \bR^d}\int_{\bR^d} |\cF G(u_j,\cdot)(\xi_j+\eta)|^2 \mu(d\xi_j) \right) \\
& \quad \quad  \times \sup_{\eta \in \bR^d} \int_{\bR^d} |\cF
G(u_n+h,\cdot)(\xi_n+\eta)|^2 \mu(d\xi_n) =\prod_{j=1}^{n-1}k(u_j)
k(u_n+h),
\end{align*}
Note that $\Gamma_{t+h} \leq \Gamma_T$ since $t+h=t' \leq T$. We observe that if $(t_1,\ldots,t_n) \in D_{t,h}$ then there exists at least one index $i$ with $t_i>t$. So,
$$D_{t,h}=\bigcup_{\rho \in S_n}\{(t_1,\ldots,t_n); 0 \leq t_{\rho(1)} \leq \ldots \leq t_{\rho(n)}, \, t<t_{\rho(n)} \leq t+h \}.$$
It follows that
\begin{align*}
B_{n}(t,h) & \leq \Gamma_T^n \sum_{\rho \in S_n} \int_t^{t+h} \int_{0<t_{\rho(1)}<\ldots<t_{\rho(n-1)}<t_{\rho(n)}} \prod_{j=1}^{n-1}k(t_{\rho(j+1)}-t_{\rho(j)}) k(t-t_{\rho(n)}+h)d{\bf t}\\
&=\Gamma_t^n n! \int_t^{t+h} \int_{0<t_1<\ldots<t_{n-1}<t_n} \prod_{j=1}^{n-1}k(t_{j+1}-t_j) k(t-t_n+h)d{\bf t}\\
&=\Gamma_T^n n! \int_t^{t+h} h_{n-1}(t_n) k(t-t_n+h)dt_n \\
& \leq \Gamma_T^n n! h_{n-1}(t+h) \int_0^{h} k(s)ds \leq \Gamma_T^n n! h_{n-1}(T) h^{\theta} \rho_{h} \leq \Gamma_T^n n! h_{n-1}(T) h^{\theta} \rho_{T}.
\end{align*}
Hence,
\begin{align}
\nonumber
\sum_{n\geq 1}(p-1)^{n/2}\left(\frac{1}{n!}B_n(t,h) \right)^{1/2} & \leq  h^{\theta/2} \rho_T \sqrt{(p-1)\Gamma_T} \sum_{n\geq 1} \sqrt{(p-1)^{n-1}\Gamma_T^{n-1} h_{n-1}(T)}\\
&=C_{T}^{(2)}h^{\theta/2},
\end{align}
where $C_T^{(2)}>0$ is a constant depending on $T$.

{\em Step 2. (increments in space)} Let $z=x'-x$. By Minkowski's inequality and \eqref{norm-p-J2},
\begin{align}
\nonumber \|u(t,x+z)-u(t,x)\|_p  & \leq \sum_{n\geq
1}\|J_n(t,x+z)-J_n(t,x)\|_p  \\
& \leq \sum_{n\geq 1} (p-1)^{n/2}\left( \frac{1}{n!}C_n(t,z)
\right)^{1/2}.
\end{align}
Recall estimate \eqref{estimate-C} for $C_n(t,h)$ that we obtained above. Using definition \eqref{def-psi-tz} of $\psi_{t,z}^{(n)}({\bf t},{\bf t})$, followed by Lemma \ref{sup-lemma} and Proposition \ref{t-increm-G}, we see that
\begin{align*}
\psi_{t,z}^{(n)}({\bf t},{\bf t}) & \leq \prod_{j=1}^{n-1} \left(\sup_{\eta \in \bR^d} \int_{\bR^d} |\cF G(u_j,\cdot)(\xi_j+\eta)|^2 \mu(d\xi_j) \right)\\
& \quad \times \sup_{\eta \in \bR^d} \int_{\bR^d} |\cF G(u_n,\cdot)(\xi_n+\eta)|^2 |1-e^{-i(\xi_n+\eta) \cdot z}|^2 \mu(d\xi_n)\\
&\leq C  |z|^{2\theta} \prod_{j=1}^{n-1}k(u_{j-1}) u_n^{-\theta}k \left(\frac{u_n}{2} \right).
\end{align*}
It follows that
\begin{align*}
C_n(t,z) & \leq C |z|^{2\theta}\Gamma_t^n n! \int_{0<t_1<\ldots<t_n<t} \prod_{j=1}^{n-1}k(t_{j+1}-t_j) (t-t_n)^{-\theta} k\left(\frac{t-t_n}{2} \right)d{\bf t}\\
&= C |z|^{2\theta}\Gamma_t^n n! \int_0^t h_{n-1}(t_n) (t-t_n)^{-\theta} k\left(\frac{t-t_n}{2} \right) dt_n\\
&\leq  C |z|^{2\theta}\Gamma_t^n n! \, h_{n-1}(t)  \rho_{t/2}.
\end{align*}
Similarly to \eqref{step-A}, we infer that
$$\sum_{n\geq 1} (p-1)^{n/2}\left(\frac{1}{n!}C_n(t,z) \right)^{1/2} \leq C_T^{(3)} |z|^{\theta},$$
where $C_T^{(3)}>0$ is a constant depending on $T$.
\end{proof}

\end{document}